\documentclass[11pt]{article}

\usepackage{amsmath,amssymb,amsthm,mathtools}
\usepackage{enumitem}
\usepackage{geometry}
\usepackage{hyperref}
\newtheorem{theorem}{Theorem}[section]
\newtheorem{lemma}[theorem]{Lemma}
\newtheorem{proposition}[theorem]{Proposition}
\newtheorem{corollary}[theorem]{Corollary}
\newtheorem{remark}[theorem]{Remark}
\newtheorem{example}[theorem]{Example}
\newtheorem{definition}[theorem]{Definition}

\newcommand{\F}{\mathbb F_2}
\newcommand{\wt}{\operatorname{wt}}
\newcommand{\bs}{\operatorname{bs}}
\newcommand{\as}{\overline{s}}
\newcommand{\preceqtwo}{\preceq_2}
\newcommand{\npreceqtwo}{\npreceq_2}

\title{Sensitivity and Block Sensitivity of Elementary Symmetric Boolean Functions of Arbitrary Degree}
\author{Yuan Li \and Jing Zhang}
\date{}

\begin{document}

\maketitle

\begin{abstract}
Let $\sigma_{n,d}$ denote the elementary symmetric Boolean function of $n$
variables and degree $d$. We completely determine the sensitivity, average
sensitivity, and block sensitivity of $\sigma_{n,d}$ for every $1\le d\le n$.
Using Lucas' theorem, we obtain a uniform binary description of the
Hamming-weight value sequence, from which the sensitivity and
average-sensitivity formulas follow and the computation of block sensitivity
reduces to at most four explicit candidates. Combining these results with the
arbitrary-degree formula for certificate complexity, we determine the exact
relations among sensitivity, block sensitivity, and certificate complexity.

We also prove a general result for symmetric Boolean functions: every
nonconstant symmetric Boolean function $f$ satisfies
\[
 \bs(f)\le \max\{s(f),C(f)-1\}.
\]
Consequently, only the three patterns
\[
 s=\bs=C,\qquad s=\bs<C,\qquad s<\bs<C
\]
can occur for nonconstant symmetric Boolean functions. For elementary
symmetric Boolean functions, we give necessary and sufficient conditions for
each of these three patterns, thereby completely classifying the relations
among $s(\sigma_{n,d})$, $\bs(\sigma_{n,d})$, and $C(\sigma_{n,d})$. In
particular, we obtain a necessary and sufficient characterization of the full
strict hierarchy
\[
 s(\sigma_{n,d})<\bs(\sigma_{n,d})<C(\sigma_{n,d}),
\]
and exhibit infinite families for which it holds.
\end{abstract}


\section{Introduction}

Sensitivity and block sensitivity are basic local complexity measures of
Boolean functions and are closely connected with decision-tree and query
complexity; they have played a central role in the study of relations among
Boolean-function complexity measures
\cite{CookDworkReischuk1986,Nisan1991,KenyonKutin2004}.
Average sensitivity records the corresponding average local behavior over
the Boolean cube.  It is therefore natural to ask for exact values of these
measures on important structured classes of Boolean functions.

Symmetric Boolean functions are among the most basic such classes: their
values depend only on Hamming weight.  Related recent work has also studied
query complexity for Boolean functions restricted to fixed Hamming-weight
slices \cite{Byramji2024}, further illustrating the role of Hamming-weight
structure in Boolean query complexity.  Elementary symmetric Boolean functions
provide a particularly natural algebraic family within this class.  For the
elementary symmetric Boolean function $\sigma_{n,d}$ of degree $d$,
\[
 \sigma_{n,d}(x)=\binom{\wt(x)}{d}\pmod 2.
\]
Consequently, Lucas' theorem converts its Hamming-weight value sequence into
a binary digit-containment problem.  This arithmetic structure makes it
possible to seek exact complexity formulas that hold uniformly in both the
number of variables $n$ and the degree $d$.

The present authors, together with Adeyeye \cite{ZhangLiAdeyeye2021}, studied
sensitivity and block sensitivity for elementary symmetric Boolean functions.
In that work we proved, among other results, that the sensitivity is $n$ for odd $d$, obtained an explicit
formula for $d=2^k$, derived formulas for several additional even degrees such
as $d=6,10,12$, and showed how further fixed degrees can be handled from their
periodic value sequences.  In the same work, we also established a general formula for the
block sensitivity of an arbitrary symmetric Boolean function in terms of the
constant runs in its value vector, and exhibited cases in which block
sensitivity is strictly greater than sensitivity.  Thus the general
symmetric-function run formula is an earlier result and is not claimed as new
here.  More recently, Mittal, Nair, and Patro \cite{MittalNairPatro2024}
studied relations among query-complexity measures for symmetric Boolean
functions.  In particular, they proved the tight constant-factor bounds
\[
 C(f)\le 2s(f),\qquad \frac{\bs(f)}{s(f)}\le\frac32,
\]
and gave constructions showing that these separations are essentially tight.
They also gave the local run formula for block sensitivity and showed that an
optimal family of sensitive blocks may be chosen so that each block consists
entirely of $1$-coordinates or entirely of $0$-coordinates.  Their results
address the magnitude of possible separations.  In contrast, one of our
general results below concerns the exact equality structure at the upper end
of the standard chain $s(f)\le\bs(f)\le C(f)$: we prove
\[
 \bs(f)\le\max\{s(f),C(f)-1\},
\]
which implies that $\bs(f)=C(f)$ forces $s(f)=\bs(f)=C(f)$.

The purpose of the present paper is to give a uniform combinatorial treatment
of sensitivity, average sensitivity, and block sensitivity for the entire
family $\{\sigma_{n,d}:1\le d\le n\}$.  Rather than treating individual
degrees through separate periodic strings or residue classes, we organize the
whole family by a binary containment structure.  The unifying observation is
obtained by writing
\[
 d=2^cD,\qquad D\ {\rm odd},\qquad h=2^c.
\]
The Hamming-weight value sequence of $\sigma_{n,d}$ is constant on blocks of
$h$ consecutive weights, while the block indexed by $t$ has value $1$ exactly
when the binary expansion of $t$ contains every $1$-bit of $D$.  Thus the
power-of-two part of $d$ determines the block length, whereas the odd part
determines the locations of the $1$-blocks.  The same binary description
drives all three complexity calculations.

This viewpoint also connects naturally with our earlier work on certificate
complexity.  The two present authors first determined certificate complexity
for odd degrees and powers of two \cite{ZhangLi2022}, and more recently
completed the arbitrary-degree case \cite{ZhangLi2026} using the same
underlying binary decomposition.  Thus \cite{ZhangLiAdeyeye2021},
\cite{ZhangLi2022}, and \cite{ZhangLi2026} are part of the same research
program by the present authors.  The objective here is to complete the
corresponding determination of sensitivity, average sensitivity, and block
sensitivity for arbitrary degree and to classify the resulting relations among
these measures and certificate complexity.

Our main contributions are as follows.
\begin{enumerate}[label=(\roman*)]
\item We obtain an exact sensitivity formula for every even degree $d$.
Together with the known odd-degree result, this determines the sensitivity of
$\sigma_{n,d}$ for every $1\le d\le n$.
\item We identify every transition of the Hamming-weight value sequence by a
binary containment condition.  The standard sensitive-edge counting identity
then gives an exact formula for the average sensitivity for every degree.
\item For block sensitivity, we start from the known run-based formula for
symmetric Boolean functions and exploit the special binary structure of
$\sigma_{n,d}$.  Internal zero-runs can be discarded, and among complete
$1$-runs only the first and last can be extremal.  Consequently the global
block sensitivity is reduced to at most four explicit candidates.
\item Comparing the exact formulas for sensitivity and block sensitivity gives
a necessary and sufficient condition for
$\bs(\sigma_{n,d})>s(\sigma_{n,d})$.  Hence the equality question is
decided for every $n$ and $d$.
\item For every nonconstant symmetric Boolean function $f$, we prove the
structural inequality
\[
 \bs(f)\le \max\{s(f),C(f)-1\}.
\]
Consequently, equality at the upper end forces equality throughout:
$\bs(f)=C(f)$ implies $s(f)=\bs(f)=C(f)$.  In particular, the pattern
$s(f)<\bs(f)=C(f)$ cannot occur for a nonconstant symmetric Boolean function.
\item Using this general theorem together with the arbitrary-degree
certificate-complexity formula from \cite{ZhangLi2026}, we completely classify
the relations among the three measures for elementary symmetric Boolean
functions.  Exactly three patterns occur,
\[
 s=\bs=C,\qquad s=\bs<C,\qquad s<\bs<C,
\]
and we give necessary and sufficient conditions for each.  For even
$d=hD$ and $n=hq+u$, the first pattern occurs exactly when
$D\preceqtwo q$ or $D\preceqtwo(q-1)$; outside this case the explicit
block-sensitivity candidates decide between the remaining two patterns.
We also exhibit explicit infinite families realizing the full strict hierarchy.
\item The odd-degree, power-of-two, and previously treated small even-degree
formulas arise naturally as special cases of the same framework.
\end{enumerate}

Taken together, these results give a uniform determination of three standard
complexity measures for elementary symmetric Boolean functions of arbitrary
degree, replacing separate degree-by-degree analyses by a single binary
structure.

\section{Preliminaries}

For integers $a\le b$, write $[a,b]=\{a,a+1,\ldots,b\}$; in particular, $[n]=[1,n]=\{1,2,\ldots,n\}$.

Let $\F=\{0,1\}$.  For $1\le d\le n$, the elementary symmetric Boolean
function of degree $d$ is
\[
 \sigma_{n,d}(x_1,\ldots,x_n)
 =
 \bigoplus_{1\le i_1<\cdots<i_d\le n}
 x_{i_1}\cdots x_{i_d}.
\]
If $\wt(x)=x_1+\cdots+x_n$, then
\begin{equation}\label{eq:value}
 \sigma_{n,d}(x)=\binom{\wt(x)}{d}\pmod 2.
\end{equation}

For a symmetric Boolean function $f$, let $v_j$ denote its value on all
vectors of Hamming weight $j$.  The sequence
\[
 v_0,v_1,\ldots,v_n
\]
will be called the Hamming-weight value sequence of $f$.

For nonnegative integers $a$ and $b$, write
\[
 a\preceqtwo b
\]
if every binary digit equal to $1$ in $a$ is also equal to $1$ in $b$.
Lucas' theorem modulo $2$ gives
\begin{equation}\label{eq:lucas}
 \binom{b}{a}\equiv 1\pmod 2
 \quad\Longleftrightarrow\quad
 a\preceqtwo b.
\end{equation}

For $x\in\F^n$, the sensitivity of $f$ at $x$, denoted $s(f;x)$, is the
number of coordinates whose individual flip changes the value of $f$.
The sensitivity is
\[
 s(f)=\max_{x\in\F^n}s(f;x),
\]
and the average sensitivity is
\[
 \as(f)=2^{-n}\sum_{x\in\F^n}s(f;x).
\]

A set $B\subseteq[n]$ is a \emph{sensitive block} for $f$ at $x$ if
flipping all coordinates in $B$ changes the value of $f$.  The block
sensitivity $\bs(f;x)$ is the maximum number of pairwise disjoint sensitive
blocks for $f$ at $x$, and
\[
 \bs(f)=\max_{x\in\F^n}\bs(f;x).
\]

A set $S\subseteq[n]$ is a \emph{certificate} for $f$ at $x$ if every
$y\in\F^n$ satisfying $y_i=x_i$ for all $i\in S$ also satisfies
$f(y)=f(x)$.  The certificate complexity of $f$ at $x$ is
\[
 C(f;x)=\min\{\lvert S\rvert:S\text{ is a certificate for }f\text{ at }x\},
\]
and the certificate complexity of $f$ is
\[
 C(f)=\max_{x\in\F^n}C(f;x).
\]
These measures satisfy the standard inequalities
\[
 s(f)\le \bs(f)\le C(f).
\]

We shall also use the following certificate formula for symmetric Boolean
functions, due to Dahiya and Mahajan \cite[Lemma~5.7]{DahiyaMahajan2023},
rewritten here in our maximal-constant-interval notation.  For completeness,
we include a short proof so that the later comparison of block sensitivity and
certificate complexity is self-contained.

\begin{lemma}[Dahiya--Mahajan: certificate complexity from constant intervals]
\label{lem:symmetric-certificate-runs}
Let $f:\{0,1\}^n\to\{0,1\}$ be symmetric, and let $[a,b]$ be the
maximal constant interval of its value vector containing the Hamming weight
$w$.  If $x$ has weight $w$, then
\[
 C(f;x)=a+(n-b)=n-(b-a+1)+1.
\]
Consequently, if $L_{\min}$ is the minimum length of a maximal constant
interval of the value vector, then
\[
 C(f)=n-L_{\min}+1.
\]
\end{lemma}

\begin{proof}
Fix an input $x$ of weight $w\in[a,b]$.  If a certificate fixes $r$ of the
$1$-coordinates of $x$ and $z$ of its $0$-coordinates, then the Hamming
weights of inputs consistent with the certificate range from $r$ to $n-z$.
For the certificate to force the value $f(x)$, this entire interval must lie
inside the maximal constant interval $[a,b]$.  Indeed, the interval $[r,n-z]$
contains $w$; if it extended beyond $[a,b]$, it would necessarily cross an
adjacent Hamming layer on which the value of $f$ differs from $f(x)$,
contradicting the certificate property.  Hence necessarily
$r\ge a$ and $n-z\le b$, so every certificate has size at least
\[
 r+z\ge a+n-b.
\]
Conversely, since $a\le w\le b$, we may fix any $a$ of the
$1$-coordinates of $x$ and any $n-b$ of its $0$-coordinates.  Every input
consistent with these fixed coordinates then has Hamming weight in $[a,b]$,
so these $a+n-b$ coordinates form a certificate.  Thus
$C(f;x)=a+n-b$.  Maximizing over $x$ is equivalent to minimizing the length
$b-a+1$ of its maximal constant interval, which gives the stated formula for
$C(f)$.
\end{proof}

We will use the following standard form of the sensitivity calculation for a
symmetric Boolean function.

Here $\mathbf 1_E$ denotes the indicator of a statement $E$: it equals
$1$ if $E$ is true and $0$ otherwise.

\begin{lemma}\label{lem:local-s}
Let $f$ be symmetric with Hamming-weight value sequence
$v_0,\ldots,v_n$.  At any vector of weight $j$,
\[
 s(f;x)
 =
 j\,\mathbf 1_{\{v_j\ne v_{j-1}\}}
 +(n-j)\,\mathbf 1_{\{v_j\ne v_{j+1}\}},
\]
where a term involving $v_{-1}$ or $v_{n+1}$ is omitted.
\end{lemma}

\begin{proof}
There are $j$ coordinates equal to $1$ and $n-j$ coordinates equal to $0$.
Flipping a $1$ decreases the Hamming weight from $j$ to $j-1$, while
flipping a $0$ increases it from $j$ to $j+1$.  Since $f$ is symmetric, the
claim follows.
\end{proof}

\section{Binary block structure}

The following two observations were first made in \cite{ZhangLi2026}.  We include their short proofs for completeness and to keep the present treatment self-contained.

\begin{lemma}[Binary block lemma]\label{lem:block}
Let
\[
 d=2^cD,\qquad D\ {\rm odd},\qquad h=2^c.
\]
Write $j=ht+v$ with $0\le v<h$.  Then
\begin{equation}\label{eq:blockcriterion}
 \binom{ht+v}{d}\equiv 1\pmod 2
 \quad\Longleftrightarrow\quad
 D\preceqtwo t.
\end{equation}
Consequently, the infinite Hamming-weight value sequence
\[
 \binom{0}{d},\binom{1}{d},\binom{2}{d},\ldots \pmod 2
\]
is constant on each block
\[
 B_t=\{ht,ht+1,\ldots,ht+h-1\},
\]
and the block $B_t$ has value $1$ if and only if $D\preceqtwo t$.
\end{lemma}

\begin{proof}
The last $c$ binary digits of $d=2^cD$ are $0$.  The last $c$ binary
digits of $ht+v$ are exactly the binary digits contributed by $v$, while
the remaining digits are those of $t$.  Hence Lucas' theorem \cite{Lucas1878} gives
\[
 2^cD\preceqtwo ht+v
 \quad\Longleftrightarrow\quad
 D\preceqtwo t,
\]
which is independent of $v$.  Equation \eqref{eq:lucas} now proves the
result.
\end{proof}

\begin{lemma}\label{lem:isolated}
Under the hypotheses of Lemma~\ref{lem:block}, every $1$-block is isolated:
if $D\preceqtwo t$, then
\[
 D\npreceqtwo(t-1)
 \qquad\hbox{and}\qquad
 D\npreceqtwo(t+1).
\]
\end{lemma}

\begin{proof}
Since $D$ is odd, its least significant binary digit is $1$.
Thus $D\preceqtwo t$ implies that $t$ is odd.  Both $t-1$ and $t+1$
are even and therefore have least significant binary digit $0$, so neither
can contain all the $1$-bits of $D$.
\end{proof}

For later use define
\[
 \chi_D(t)=
 \begin{cases}
 1,&D\preceqtwo t,\\
 0,&D\npreceqtwo t,
 \end{cases}
\]
and define the transition set
\begin{equation}\label{eq:transition-set}
 \mathcal T_D(q)
 =
 \{\,t:1\le t\le q,\ \chi_D(t)\ne\chi_D(t-1)\,\}.
\end{equation}
By Lemma~\ref{lem:isolated},
\begin{equation}\label{eq:transition-equivalent}
 t\in\mathcal T_D(q)
 \quad\Longleftrightarrow\quad
 D\preceqtwo t\ \hbox{or}\ D\preceqtwo(t-1).
\end{equation}

\section{Sensitivity for arbitrary degree}

For odd degree the sensitivity is already known.  The following theorem was first proved in \cite{ZhangLiAdeyeye2021}; we recall it here because it is used repeatedly below.

\begin{theorem}[Odd degree]\label{thm:odd}
If $d$ is odd and $1\le d\le n$, then
\[
 s(\sigma_{n,d})=n.
\]
\end{theorem}

We now treat every even degree at once.

\begin{definition}\label{def:R}
For odd $D$ and $q\ge D$, define
\[
 R_D(q)=\max\{\,r\le q:D\preceqtwo r\,\}.
\]
\end{definition}

The set in Definition~\ref{def:R} is nonempty because $D\preceqtwo D$.

\begin{theorem}[Sensitivity for arbitrary even degree]\label{thm:main-s}
Let $d$ be even, $2\le d\le n$, and write
\[
 d=hD,\qquad h=2^{\nu_2(d)},\qquad D\ {\rm odd}.
\]
Write
\[
 n=hq+u,\qquad 0\le u<h.
\]
Then
\begin{equation}\label{eq:main-s}
 \boxed{
 s(\sigma_{n,d})
 =
 \max\left\{
 n-d+1,\,
 h\min\{q,R_D(q)+1\}
 \right\}.}
\end{equation}
Equivalently, if
\[
 \tau_D(q)=\max\mathcal T_D(q),
\]
then
\begin{equation}\label{eq:main-s-transition}
 s(\sigma_{n,d})
 =
 \max\{\,n-d+1,\ h\tau_D(q)\,\}.
\end{equation}
\end{theorem}

\begin{proof}
By Lemma~\ref{lem:block}, the value of $\sigma_{n,d}$ is constant on each
$h$-block, and the $t$-block has value $1$ exactly when
$D\preceqtwo t$.  Consequently a change in the Hamming-weight value
sequence can occur only at a boundary between two consecutive $h$-blocks.
The boundary between the $(t-1)$-block and the $t$-block lies between
weights $ht-1$ and $ht$, and by definition it is a transition precisely
when $t\in\mathcal T_D(q)$.  Thus every nonzero local sensitivity must
come from one of these transition boundaries.

We first locate the earliest transition.  If $t<D$, then $D\preceqtwo t$
is impossible, since binary containment implies $t\ge D$.  On the other
hand $D\preceqtwo D$.  Hence the blocks with indices below $D$ have value
$0$, while the $D$-block has value $1$, so the first transition occurs at
$t=D$.  It is therefore between weights
\[
 d-1=hD-1
 \quad\hbox{and}\quad
 d=hD.
\]
At weight $d-1$, changing any one of the $n-d+1$ zero coordinates raises
the weight to $d$ and crosses this transition.  Hence
\[
 s(\sigma_{n,d};d-1)=n-d+1.
\]

Now let $\tau=\max\mathcal T_D(q)$ be the index of the last transition.
That transition lies between weights $h\tau-1$ and $h\tau$.  At weight
$h\tau$, changing any one of the $h\tau$ coordinates equal to $1$ lowers
the weight to $h\tau-1$ and crosses the transition.  Thus
\[
 s(\sigma_{n,d};h\tau)=h\tau.
\]

It remains to see that no interior transition can do better.  Since $d$ is
even, $h\ge2$.  Therefore the two levels $ht-1$ and $ht$ adjacent to the
boundary at $t$ cannot simultaneously be adjacent to a second $h$-block
boundary.  If $t\in\mathcal T_D(q)$, Lemma~\ref{lem:local-s} therefore
gives exactly
\[
 s(\sigma_{n,d};ht-1)=n-ht+1,
 \qquad
 s(\sigma_{n,d};ht)=ht.
\]
As $t$ increases, $n-ht+1$ strictly decreases and $ht$ strictly increases.
Hence, among all transition boundaries, the largest value of the first
quantity occurs at the first transition $t=D$, while the largest value of
the second occurs at the last transition $t=\tau$.  Thus
\[
 s(\sigma_{n,d})=\max\{n-d+1,h\tau\}.
\]

Finally we express $\tau$ in terms of $R_D(q)$.  Put $r=R_D(q)$, so $r$ is
the largest block index at most $q$ whose block has value $1$.  If $r=q$,
then the $q$-block has value $1$.  Since $D$ is odd,
Lemma~\ref{lem:isolated} shows that the $(q-1)$-block has value $0$; hence
the boundary at $q$ is a transition and $\tau=q$.  If $r<q$, maximality of
$r$ implies that the $(r+1)$-block has value $0$, whereas the $r$-block has
value $1$.  Thus the boundary at $r+1$ is a transition.  Moreover there is
no later $1$-block, so no later transition can occur.  Hence
$\tau=r+1$.  Combining the two cases gives
\[
 \tau=\min\{q,R_D(q)+1\},
\]
which substituted above yields \eqref{eq:main-s}.
\end{proof}

\begin{corollary}\label{cor:end-transition}
Under the hypotheses of Theorem~\ref{thm:main-s}, if
\[
 D\preceqtwo q
 \quad\hbox{or}\quad
 D\preceqtwo(q-1),
\]
then
\[
 s(\sigma_{n,d})
 =
 \max\{n-d+1,hq\}.
\]
In particular, since $hq=n-u$,
\[
 s(\sigma_{n,d})
 =
 \max\{n-d+1,n-u\}.
\]
\end{corollary}

\begin{proof}
The condition means exactly that the boundary at $hq$ is a transition, so
$\tau_D(q)=q$.  Apply Theorem~\ref{thm:main-s}.
\end{proof}

\begin{corollary}[Power-of-two degree]\label{cor:power2}
The following result was previously proved as Theorem~3.19 in
\cite{ZhangLiAdeyeye2021}.  If $d=2^k$ and $n\ge d$, then
\[
 s(\sigma_{n,2^k})
 =
 2^k\left\lfloor\frac{n}{2^k}\right\rfloor .
\]
\end{corollary}

\begin{proof}
Here $D=1$ and $h=2^k$.  Exactly one of $q$ and $q-1$ is odd, so
$1\preceqtwo q$ or $1\preceqtwo(q-1)$.  Hence the last boundary is a
transition and $h\tau_D(q)=hq$.

It remains to compare $hq$ with $n-h+1$.  Since $n=hq+u$ and
$0\le u<h$,
\[
 hq-(n-h+1)=h-u-1\ge0.
\]
Theorem~\ref{thm:main-s} therefore gives
\[
 s(\sigma_{n,2^k})=hq
 =2^k\left\lfloor\frac{n}{2^k}\right\rfloor.
\]
\end{proof}

\subsection{Examples}

\begin{example}[Degree $6$]
Let $d=6=2\cdot3$, so $h=2$ and $D=3$.
The $1$-blocks are indexed by integers whose last two binary digits are
$11$, namely
\[
 3,7,11,15,\ldots.
\]
Write $n=8a+r$, $0\le r\le7$.  The last transition occurs at weight
$8a$ when $0\le r\le5$, at weight $n$ when $r=6$, and at weight
$n-1$ when $r=7$.  Theorem~\ref{thm:main-s} therefore gives
\[
 s(\sigma_{n,6})
 =
 \begin{cases}
 8a,&0\le r\le5,\\
 n,&r=6,\\
 n-1,&r=7.
 \end{cases}
\]
This is exactly the previously obtained degree-$6$ formula in
\cite{ZhangLiAdeyeye2021}, now recovered from the arbitrary-degree theorem.
\end{example}

\begin{example}[Degree $10$]
Let $d=10=2\cdot5$, so $h=2$ and $D=5=(101)_2$.
The $1$-blocks are precisely those $t$ whose binary expansion contains the
first and third least significant $1$-bits.  Theorem~\ref{thm:main-s}
therefore replaces the separate residue-class analysis modulo $16$ by the
single binary test $5\preceqtwo t$.
\end{example}

\begin{example}[Degree $12$]
Let $d=12=4\cdot3$, so $h=4$ and $D=3$.
The value sequence is constant on blocks of four consecutive Hamming
weights, and the $t$-block has value $1$ exactly when
$3\preceqtwo t$.  Again Theorem~\ref{thm:main-s} gives the sensitivity
directly.
\end{example}

\section{Average sensitivity for arbitrary degree}

The same transition set, together with the standard interpretation of
average sensitivity as twice the normalized number of sensitive hypercube
edges, gives a compact formula for elementary symmetric Boolean functions.

\begin{theorem}\label{thm:average}
Let $1\le d\le n$, write
\[
 d=hD,\qquad h=2^{\nu_2(d)},\qquad D\ {\rm odd},
\]
and write $n=hq+u$, $0\le u<h$.  Then
\begin{equation}\label{eq:average}
 \boxed{
 \as(\sigma_{n,d})
 =
 2^{1-n}
 \sum_{t\in\mathcal T_D(q)}
 ht\binom{n}{ht}.}
\end{equation}
Equivalently,
\begin{equation}\label{eq:average2}
 \as(\sigma_{n,d})
 =
 2^{1-n}
 \sum_{t\in\mathcal T_D(q)}
 (n-ht+1)\binom{n}{ht-1}.
\end{equation}
\end{theorem}

\begin{proof}
Recall that the average sensitivity is
\[
 \as(\sigma_{n,d})=2^{-n}\sum_{x\in\F^n}s(\sigma_{n,d};x).
\]
It is convenient to count the same quantity by edges of the Boolean cube.
An edge joining a vector of weight $k-1$ to a vector of weight $k$ is
sensitive exactly when the two Hamming-weight values differ, that is, when
$v_{k-1}\ne v_k$.

Fix such a transition between levels $k-1$ and $k$.  Every vector of weight
$k$ has exactly $k$ neighbors of weight $k-1$, so the number of edges
between these two levels is
\[
 k\binom{n}{k}.
\]
Equivalently, every vector of weight $k-1$ has $n-k+1$ neighbors of weight
$k$, giving the identity
\[
 k\binom{n}{k}
 =
 (n-k+1)\binom{n}{k-1}.
\]
Each sensitive edge is counted once in the sensitivity of each endpoint.
Consequently its total contribution to
$\sum_x s(\sigma_{n,d};x)$ is $2$, and hence
\[
 \sum_{x\in\F^n}s(\sigma_{n,d};x)
 =
 2\sum_{\substack{1\le k\le n\\ v_{k-1}\ne v_k}}
 k\binom{n}{k}.
\]

We now use the binary block structure.  By Lemma~\ref{lem:block}, the value
sequence is constant inside every $h$-block.  Thus transitions can occur
only at weights $k=ht$.  The boundary at $ht$ is a transition exactly when
the two adjacent block values differ, which is precisely the condition
$t\in\mathcal T_D(q)$.  Therefore the preceding sum becomes
\[
 \sum_{x\in\F^n}s(\sigma_{n,d};x)
 =2\sum_{t\in\mathcal T_D(q)}
 ht\binom{n}{ht}.
\]
This description also covers the case in which the final $q$-block is
incomplete: if $hq\le n$, its left boundary between weights $hq-1$ and
$hq$ is present and is included when $q\in\mathcal T_D(q)$, while there is
no boundary after weight $n$ and hence no extra term.

Dividing by $2^n$ gives \eqref{eq:average}.  Finally, applying
\[
 ht\binom{n}{ht}
 =
 (n-ht+1)\binom{n}{ht-1}
\]
term by term gives the equivalent expression \eqref{eq:average2}.
\end{proof}

\begin{remark}
Formula \eqref{eq:average} is valid for both odd and even degrees.  When
$d$ is odd, $h=1$; the difference from ordinary sensitivity is that the
average sensitivity depends on all transitions rather than only the
largest local value.
\end{remark}

\section{Block sensitivity from binary constant intervals}

We now apply the binary block structure to block sensitivity.  A general
run-based formula for block sensitivity of symmetric Boolean functions was
obtained in \cite{ZhangLiAdeyeye2021}.  A closely related local formulation,
together with the reduction to sensitive blocks consisting entirely of
$1$-coordinates or entirely of $0$-coordinates, also appears as Lemmas 8 and
9 of Mittal, Nair, and Patro \cite{MittalNairPatro2024}.  We include the
local argument below for completeness and to fix notation; it is not claimed
as a new result.  The new work in this section is the use of the special
binary structure of $\sigma_{n,d}$ to obtain an explicit arbitrary-degree
reduction.

Let $f$ be a symmetric Boolean function and suppose that
\[
 [A,B]\subseteq\{0,1,\ldots,n\}
\]
is the maximal constant interval containing the Hamming weight $w$.

\begin{lemma}\label{lem:local-bs}
For $A\le w\le B$,
\begin{equation}\label{eq:local-bs}
 \bs(f;w)
 =
 \left\lfloor\frac{w}{w-A+1}\right\rfloor
 +
 \left\lfloor\frac{n-w}{B-w+1}\right\rfloor ,
\end{equation}
where the first term is $0$ when $A=0$ and the second term is $0$ when
$B=n$.
\end{lemma}

\begin{proof}
Fix a vector $x$ of Hamming weight $w$, with $A\le w\le B$, and let a
sensitive block contain $a$ coordinates equal to $1$ in $x$ and $b$
coordinates equal to $0$ in $x$.  After flipping that block, the Hamming
weight becomes
\[
 w-a+b.
\]
Since $[A,B]$ is the maximal constant interval containing $w$, sensitivity
of the block means that either
\[
 w-a+b\le A-1
\]
or
\[
 w-a+b\ge B+1.
\]

Suppose first that
\[
 w-a+b\le A-1.
\]
Then
\[
 a-b\ge w-A+1,
\]
and therefore
\[
 a\ge w-A+1.
\]
Choose any $w-A+1$ of the $a$ coordinates equal to $1$ that belong to the
block and flip only those coordinates.  The resulting Hamming weight is
exactly $A-1$, so the value of the function changes.  Hence every sensitive
block that moves below the interval can be replaced by a pure downward
sensitive block consisting of exactly $w-A+1$ coordinates equal to $1$.

Likewise, if
\[
 w-a+b\ge B+1,
\]
then
\[
 b-a\ge B-w+1,
\]
and hence
\[
 b\ge B-w+1.
\]
Choosing any $B-w+1$ of the zero-coordinates in the block gives a pure
upward sensitive block that moves the Hamming weight exactly to $B+1$.

Now start with any family of pairwise disjoint sensitive blocks.
Replace every downward block by such a pure downward subblock and every
upward block by such a pure upward subblock.  Because each replacement is
a subset of the original block, pairwise disjointness is preserved.  Thus
there is an optimal family in which every sensitive block is either

\begin{itemize}
\item a pure downward block of size $w-A+1$, using only coordinates equal
to $1$, or
\item a pure upward block of size $B-w+1$, using only coordinates equal
to $0$.
\end{itemize}

The maximum possible number of pairwise disjoint downward blocks is
therefore
\[
 \left\lfloor\frac{w}{w-A+1}\right\rfloor,
\]
and the maximum possible number of pairwise disjoint upward blocks is
\[
 \left\lfloor\frac{n-w}{B-w+1}\right\rfloor.
\]
These two families use disjoint types of coordinates, so they can be chosen
simultaneously.  Hence
\[
 \bs(f;w)
 =
 \left\lfloor\frac{w}{w-A+1}\right\rfloor
 +
 \left\lfloor\frac{n-w}{B-w+1}\right\rfloor.
\]

If $A=0$, no sensitive block can move below the interval, so the first term
is omitted.  If $B=n$, no sensitive block can move above the interval, so
the second term is omitted.
\end{proof}

\begin{proposition}\label{prop:run-bs}
Let $[A,B]$ be a maximal constant interval of a symmetric Boolean function
$f$, and put $L=B-A+1$.  The largest block sensitivity attained at a
Hamming weight in this interval occurs at one of its two endpoints.
Consequently its contribution to $\bs(f)$ is
\begin{equation}\label{eq:run-bs}
 M(A,B)
 =
 \max\left\{
 A+\left\lfloor\frac{n-A}{L}\right\rfloor,\,
 n-B+\left\lfloor\frac{B}{L}\right\rfloor
 \right\},
\end{equation}
with the obvious endpoint interpretation when $A=0$ or $B=n$.
\end{proposition}

\begin{proof}
Let
\[
 w=A+x,\qquad 0\le x\le L-1,
\]
where $L=B-A+1$.  By Lemma~\ref{lem:local-bs},
\[
 F(x):=\bs(f;A+x)
 =
 \left\lfloor\frac{A+x}{x+1}\right\rfloor
 +
 \left\lfloor\frac{n-A-x}{L-x}\right\rfloor.
\]
For an interior run, define
\[
 r(x)
 =
 \frac{A+x}{x+1}
 +
 \frac{n-A-x}{L-x}
 =
 2+\frac{A-1}{x+1}
  +\frac{n-B-1}{L-x}.
\]
The function $r$ is convex on $[0,L-1]$, since
\[
 r''(x)
 =
 \frac{2(A-1)}{(x+1)^3}
 +
 \frac{2(n-B-1)}{(L-x)^3}
 \ge0.
\]
Hence
\[
 r(x)\le\max\{r(0),r(L-1)\}.
\]
Since $F(x)\le r(x)$ and $F(x)$ is an integer,
\[
 F(x)
 \le
 \left\lfloor\max\{r(0),r(L-1)\}\right\rfloor.
\]
At the endpoints, one summand is an integer, so
\[
 \lfloor r(0)\rfloor=F(0),
 \qquad
 \lfloor r(L-1)\rfloor=F(L-1).
\]
Therefore
\[
 F(x)\le\max\{F(0),F(L-1)\}.
\]
Thus the maximum is attained at $w=A$ or $w=B$.  Evaluating
Lemma~\ref{lem:local-bs} at these endpoints gives
\[
 M(A,B)
 =
 \max\left\{
 A+\left\lfloor\frac{n-A}{L}\right\rfloor,\,
 n-B+\left\lfloor\frac{B}{L}\right\rfloor
 \right\}.
\]
The cases $A=0$ or $B=n$ are one-sided and follow directly from
Lemma~\ref{lem:local-bs}.
\end{proof}

\begin{proposition}[Run formula for arbitrary degree]\label{prop:bs-run}
Let $1\le d\le n$ and write
\[
 d=hD,\qquad h=2^{\nu_2(d)},\qquad D\ {\rm odd},
\]
and $n=hq+u$, $0\le u<h$.
Form the finite block-index sequence
\[
 \chi_D(0),\chi_D(1),\ldots,\chi_D(q),
 \qquad
 \chi_D(t)=\mathbf 1_{\{D\preceqtwo t\}},
\]
where the final $q$-block contains only the weights
$hq,\ldots,hq+u$.
Let
\[
 [A_1,B_1],\ldots,[A_m,B_m]
\]
be the maximal constant Hamming-weight intervals obtained by merging
consecutive blocks having the same value.  Then
\begin{equation}\label{eq:bs-general}
 \boxed{
 \bs(\sigma_{n,d})
 =
 \max_{1\le i\le m}
 M(A_i,B_i),}
\end{equation}
where $M(A,B)$ is given by \eqref{eq:run-bs}.
\end{proposition}

\begin{proof}
Lemma~\ref{lem:block} determines the complete Hamming-weight value sequence
of $\sigma_{n,d}$ from the binary condition $D\preceqtwo t$.  Hence it
determines all maximal constant intervals $[A_i,B_i]$.  Apply
Proposition~\ref{prop:run-bs} to each interval and take the maximum.
\end{proof}

We now simplify Proposition~\ref{prop:bs-run} considerably.  The key point
is that internal zero-runs are dominated by their neighboring $1$-runs, and
among complete $1$-runs only the first and the last can be extremal.

\begin{lemma}\label{lem:zero-run-dominated}
Suppose $r<s\le q$ are consecutive elements of the set
\[
 \{\,t\ge0:D\preceqtwo t\,\}.
\]
Thus
\[
 D\preceqtwo r,\qquad D\preceqtwo s,
\]
and there is no integer $t$ with $r<t<s$ for which $D\preceqtwo t$.
Then the zero-run between the corresponding $1$-blocks is
\[
 [h(r+1),hs-1].
\]
The block-sensitivity contribution of this zero-run is no larger than the
maximum of the contributions of the two adjacent $1$-runs.
\end{lemma}

\begin{proof}
Put
\[
 g=s-r-1\ge1.
\]
Then the zero-run has length $hg$.

At its left endpoint $A=h(r+1)$, Proposition~\ref{prop:run-bs} gives the
candidate
\[
 Z_L
 =
 h(r+1)+
 \left\lfloor
 \frac{n-h(r+1)}{hg}
 \right\rfloor.
\]
Since $n=hq+u$ with $0\le u<h$ and
$s=r+g+1$, we have
\[
 n-h(r+1)
 =
 h(q-s+g)+u.
\]
Therefore
\[
 \left\lfloor
 \frac{n-h(r+1)}{hg}
 \right\rfloor
 \le q-s+1.
\]
It follows that
\[
 Z_L
 \le h(r+1)+q-s+1.
\]
On the other hand, the left-endpoint contribution of the following
$1$-run, indexed by $s$, is at least
\[
 q+(h-1)s.
\]
Indeed,
\[
 q+(h-1)s-\bigl(h(r+1)+q-s+1\bigr)
 =
 hg-1\ge0.
\]
Hence the left endpoint of the zero-run is dominated by the following
$1$-run.

At the right endpoint $B=hs-1$, the zero-run candidate is
\[
 Z_R
 =
 n-hs+1+
 \left\lfloor
 \frac{hs-1}{hg}
 \right\rfloor.
\]
Because $g\ge1$,
\[
 \left\lfloor
 \frac{hs-1}{hg}
 \right\rfloor
 \le
 \left\lfloor
 \frac{hs-1}{h}
 \right\rfloor
 =
 s-1
 =
 r+g
 \le r+hg.
\]
Thus
\[
 Z_R
 \le
 n-hs+1+r+hg
 =
 n-h+1-(h-1)r.
\]
The right-hand side is exactly the right-endpoint contribution of the
preceding $1$-run indexed by $r$.

Therefore both endpoint candidates of the internal zero-run are dominated
by adjacent $1$-runs.  By Proposition~\ref{prop:run-bs}, the zero-run
cannot determine the global block sensitivity.
\end{proof}

\begin{definition}\label{def:final-candidate}
Under the notation
\[
 d=hD,\qquad n=hq+u,\qquad 0\le u<h,
\]
define the final-run candidate $E_{n,d}$ by
\[
 E_{n,d}
 =
 \begin{cases}
 \displaystyle
 \max\left\{hq,\left\lfloor\frac{n}{u+1}\right\rfloor\right\},
 & D\preceqtwo q,\\[4mm]
 \displaystyle
 \max\left\{A,\left\lfloor\frac{n}{n-A+1}\right\rfloor\right\},
 & D\npreceqtwo q,
 \end{cases}
\]
where, in the second case,
\[
 r=R_D(q),\qquad A=h(r+1).
\]
\end{definition}

The endpoint case $d=n$ is immediate from the previously established identity $s(\sigma_{n,n})=n$ \cite{ZhangLiAdeyeye2021} and the standard inequalities $s(f)\le\bs(f)\le n$; hence $\bs(\sigma_{n,n})=n$.

\begin{theorem}[Explicit block sensitivity for arbitrary degree]
\label{thm:bs-explicit}
Let $1\le d\le n$, and write
\[
 d=hD,\qquad h=2^{\nu_2(d)},\qquad D\ {\rm odd},
\]
and
\[
 n=hq+u,\qquad 0\le u<h.
\]
Put
\[
 I=n-d+1.
\]
If $q>D$, also put
\[
 \rho=R_D(q-1),\qquad
 P=n-h+1-(h-1)D,\qquad
 Q=q+(h-1)\rho.
\]
Then
\[
 \boxed{
 \bs(\sigma_{n,d})
 =
 \begin{cases}
 \max\{I,E_{n,d}\},&q=D,\\[2mm]
 \max\{I,E_{n,d},P,Q\},&q>D.
 \end{cases}}
\]
Thus the block sensitivity is determined by at most four candidates,
coming from the initial run, the final run, the first complete $1$-run,
and the last complete $1$-run.
\end{theorem}

\begin{proof}
The initial constant interval is
\[
 [0,d-1].
\]
By Proposition~\ref{prop:run-bs}, its contribution is
\[
 n-d+1=I.
\]

Now consider a complete $1$-run indexed by an integer $r<q$ satisfying
$D\preceqtwo r$.  Since every $1$-block is isolated, this run is exactly
\[
 [hr,h(r+1)-1]
\]
and has length $h$.  Proposition~\ref{prop:run-bs} gives its contribution
as
\[
 \max\left\{
 hr+\left\lfloor\frac{n-hr}{h}\right\rfloor,\,
 n-h(r+1)+1+\left\lfloor\frac{h(r+1)-1}{h}\right\rfloor
 \right\}.
\]
Because $n=hq+u$ with $0\le u<h$,
\[
 \left\lfloor\frac{n-hr}{h}\right\rfloor=q-r
\]
and
\[
 \left\lfloor\frac{h(r+1)-1}{h}\right\rfloor=r.
\]
Hence the contribution of this $1$-run is
\[
 \max\left\{
 q+(h-1)r,\,
 n-h+1-(h-1)r
 \right\}.
\]
The first expression is increasing in $r$, while the second is decreasing
in $r$.  Therefore, among all complete $1$-runs, only the first and the
last can matter.

The first complete $1$-run occurs at $r=D$, and its decreasing-end
candidate is
\[
 P=n-h+1-(h-1)D.
\]
If $q>D$, the last complete $1$-run occurs at
\[
 \rho=R_D(q-1),
\]
and its increasing-end candidate is
\[
 Q=q+(h-1)\rho.
\]

By Lemma~\ref{lem:zero-run-dominated}, every internal zero-run is dominated
by one of its neighboring $1$-runs and may be discarded from the global
maximum.

It remains only to consider the final constant interval.

If $D\preceqtwo q$, then the final interval is the truncated $1$-block
\[
 [hq,n],
\]
whose length is $u+1$.  Proposition~\ref{prop:run-bs} gives
\[
 E=
 \max\left\{
 hq,\,
 \left\lfloor\frac{n}{u+1}\right\rfloor
 \right\}.
\]

If $D\npreceqtwo q$, let $r=R_D(q)$.  Then the last $1$-block occurs at
$r$, and the final zero-run begins at
\[
 A=h(r+1)
\]
and ends at $n$.  Its length is
\[
 L=n-A+1.
\]
Again Proposition~\ref{prop:run-bs} gives
\[
 E=
 \max\left\{
 A,\,
 \left\lfloor\frac{n}{L}\right\rfloor
 \right\}.
\]

Thus, when $q=D$, there is no complete $1$-run strictly before the final
block, so only $I$ and $E_{n,d}$ are needed.  When $q>D$, the only additional
possible extrema are $P$ and $Q$.  This proves the theorem.
\end{proof}

\begin{remark}
Every complete $1$-run has length exactly $h$ by
Lemma~\ref{lem:isolated}.  Zero-runs are obtained from gaps between
successive integers $t$ satisfying $D\preceqtwo t$.  Thus
Proposition~\ref{prop:bs-run} reduces the block-sensitivity problem for arbitrary
degree to the binary spacing of the supermasks of $D$.
\end{remark}

\begin{corollary}
If $d$ is odd, then
\[
 \bs(\sigma_{n,d})=n.
\]
This result was previously established in \cite{ZhangLiAdeyeye2021}.
\end{corollary}

\begin{proof}
By Theorem~\ref{thm:odd}, $s(\sigma_{n,d})=n$, and for every Boolean
function
\[
 s(f)\le \bs(f)\le n.
\]
\end{proof}

\begin{corollary}[Equality versus strict separation]
\label{cor:equality-separation}
For every $1\le d\le n$, the formulas above decide whether
\[
 s(\sigma_{n,d})=\bs(\sigma_{n,d})
 \quad\text{or}\quad
 s(\sigma_{n,d})<\bs(\sigma_{n,d}).
\]
If $d$ is odd, equality always holds.  If $d$ is even, write
\[
 d=hD,\qquad h=2^{\nu_2(d)},\qquad D\ \mathrm{odd},
 \qquad n=hq+u,\quad 0\le u<h,
\]
and set
\[
 \tau=\min\{q,R_D(q)+1\},\qquad
 S=\max\{n-d+1,h\tau\}.
\]
Then $S=s(\sigma_{n,d})$.  If $q>D$, put
\[
 \rho=R_D(q-1),\qquad
 P=n-h+1-(h-1)D,\qquad Q=q+(h-1)\rho.
\]
Then
\[
 \bs(\sigma_{n,d})>s(\sigma_{n,d})
\]
if and only if at least one of the following explicit inequalities holds:
\[
 E_{n,d}>S\quad\text{or}\quad P>S\quad\text{or}\quad Q>S,
\]
where the last two are present only when $q>D$.
Equivalently, equality holds precisely when every applicable one of these
three inequalities is reversed.  Thus equality versus strict separation is
decided entirely by the binary supermask data $R_D(q)$ and $R_D(q-1)$ and
by elementary arithmetic in $h,q,u,D$.
\end{corollary}

\begin{proof}
For odd $d$, equality follows from Theorem~\ref{thm:odd} and the preceding
corollary.  For even $d$, Theorem~\ref{thm:main-s} gives $S$, while
Theorem~\ref{thm:bs-explicit} gives the block sensitivity.  The initial-run
candidate $I=n-d+1$ is already one of the two candidates defining $S$, and
therefore cannot produce a strict separation.  The only remaining candidates
are exactly $E_{n,d}$ and, when $q>D$, $P$ and $Q$.  Comparing them with $S$
gives the criterion.
\end{proof}

\begin{remark}[Useful binary simplifications]
\label{rem:separation-simplifications}
The preceding criterion can often be checked without evaluating all three
candidates.  First,
\[
 P-(n-d+1)=D-h,
\]
so the first complete $1$-run can beat the initial-run sensitivity candidate
only when $D>h$.  Second, if $D\preceqtwo q$, then $\tau=q$ and the final-run
candidate satisfies $E_{n,d}\le hq\le S$; hence the final run cannot create a
strict separation.  Third, if $D\npreceqtwo q$ and
$\rho=R_D(q)=R_D(q-1)<q$, then $\tau=\rho+1$ and
\[
 Q-h\tau=q-\rho-h.
\]
Thus the last complete $1$-run can beat the last-transition sensitivity
candidate only when the terminal gap $q-\rho$ between $q$ and the last
supermask of $D$ is larger than $h$.  These identities expose directly how
strict separation is governed by the binary spacing of the supermasks of the
odd part $D$.
\end{remark}

The strict inequalities found for particular elementary symmetric Boolean
functions in \cite{ZhangLiAdeyeye2021} are therefore instances of a question
that is now decidable for arbitrary degree.  In particular, that paper already
gives fixed-degree infinite families, for example $d=6$ with
$n\equiv4,5\pmod 8$.  The arbitrary-degree formulas also yield infinite
families in which the degree itself grows.

\begin{corollary}[An infinite growing-degree separation family]
\label{cor:s-bs-growing-family}
For every integer $k\ge2$, put
\[
 d_k=2^{k+1}-2,
 \qquad
 n_k=2^{k+1}+4.
\]
Then
\[
 s(\sigma_{n_k,d_k})=2^{k+1}
 <2^{k+1}+1=\bs(\sigma_{n_k,d_k}).
\]
Consequently, there are infinitely many elementary symmetric Boolean
functions of unbounded degree for which sensitivity is strictly smaller than
block sensitivity.
\end{corollary}

\begin{proof}
Write
\[
 d_k=2D,
 \qquad D=2^k-1,
\]
so $h=2$.  Since
\[
 n_k=2q,
 \qquad q=2^k+2=D+3,
 \qquad u=0,
\]
and $D=(11\cdots1)_2$ has its lowest $k$ binary digits all equal to $1$,
the supermasks of $D$ are precisely the integers whose lowest $k$ bits are
all $1$.  Because $k\ge2$, the last such integer not exceeding $q$ is $D$.
Thus
\[
 R_D(q)=R_D(q-1)=D,
 \qquad \tau=D+1=2^k.
\]
Theorem~\ref{thm:main-s} gives
\[
 s(\sigma_{n_k,d_k})
 =\max\{n_k-d_k+1,2\tau\}
 =\max\{7,2^{k+1}\}
 =2^{k+1}.
\]
For block sensitivity, Theorem~\ref{thm:bs-explicit} has the candidates
\[
 I=7,
 \qquad
 P=n_k-1-D=2^k+4,
 \qquad
 Q=q+D=2^{k+1}+1.
\]
Since $D\npreceqtwo q$, the final interval begins at
\[
 A=2(D+1)=2^{k+1}
\]
and has length $n_k-A+1=5$.  Hence
\[
 E_{n_k,d_k}
 =\max\left\{2^{k+1},
 \left\lfloor\frac{2^{k+1}+4}{5}\right\rfloor\right\}
 =2^{k+1}.
\]
For $k\ge2$, the largest of the four candidates is therefore $Q$, and
\[
 \bs(\sigma_{n_k,d_k})=2^{k+1}+1.
\]
This proves the strict separation.
\end{proof}

\subsection{A general relation among $s$, $\bs$, and $C$}

We first record a general observation that applies to every nonconstant
symmetric Boolean function, not only to elementary symmetric functions.  It
shows that one of the four formally possible equality/strictness patterns among
$s$, $\bs$, and $C$ can never occur.

\begin{theorem}[A general symmetric-function inequality]
\label{thm:symmetric-bs-c}
Let $f:\{0,1\}^n\to\{0,1\}$ be a nonconstant symmetric Boolean function.
Then
\[
 \bs(f)\le \max\{s(f),C(f)-1\}.
\]
Consequently,
\[
 \bs(f)=C(f)\quad\Longrightarrow\quad s(f)=\bs(f)=C(f).
\]
In particular, the pattern
\[
 s(f)<\bs(f)=C(f)
\]
cannot occur for a nonconstant symmetric Boolean function.
\end{theorem}

\begin{proof}
Let $(v_0,v_1,\ldots,v_n)$ be the value vector of $f$, and let $[a,b]$ be a
maximal constant interval containing the Hamming weight $w$.  Put
$L=b-a+1$.

First suppose that $0<a\le b<n$.  Put
\[
 p=w-a+1,\qquad q=b-w+1.
\]
As in the pure-block reduction of \cite[Lemma~9]{MittalNairPatro2024},
a sensitive block that moves the Hamming weight below $a$ has net downward
change at least $p$.  If such a block contains $r$ flips of type
$1\to0$ and $t$ flips of type $0\to1$, then $r-t\ge p$, and hence
$r\ge p$.  Thus every downward sensitive block uses at least $p$ of the
$w$ coordinates that are $1$ in the original input, so a pairwise disjoint
family contains at most $\lfloor w/p\rfloor$ downward blocks.  Similarly,
every upward sensitive block uses at least $q$ of the $n-w$ original
$0$-coordinates, so there are at most $\lfloor(n-w)/q\rfloor$ pairwise
disjoint upward blocks.  Both bounds are attained simultaneously by taking
pure blocks consisting respectively of exactly $p$ original $1$-coordinates
and exactly $q$ original $0$-coordinates.  Therefore
\[
 \bs(f;w)=
 \left\lfloor\frac{w}{p}\right\rfloor+
 \left\lfloor\frac{n-w}{q}\right\rfloor.
\]
Since $p+q=L+1$, we have
\[
 \bs(f;w)=2+
 \left\lfloor\frac{a-1}{p}\right\rfloor+
 \left\lfloor\frac{n-b-1}{q}\right\rfloor.                 \tag{*}
\]
If $L\ge2$, then $p+q\ge3$, so at least one of $p,q$ is at least $2$.
Assume first that $a\ge2$ and $b\le n-2$.  If $p\ge2$, then
\[
 \left\lfloor\frac{a-1}{p}\right\rfloor\le a-2,
\]
whereas
\[
 \left\lfloor\frac{n-b-1}{q}\right\rfloor\le n-b-1.
\]
If instead $q\ge2$, the symmetric estimates give the same total bound.
Thus in either case
\[
 \left\lfloor\frac{a-1}{p}\right\rfloor+
 \left\lfloor\frac{n-b-1}{q}\right\rfloor
 \le (a-1)+(n-b-1)-1.
\]
Substitution in $(*)$ yields
\[
 \bs(f;w)\le
 2+(a-1)+(n-b-1)-1=n-b+a-1=n-L.
\]
If instead $a=1$, maximality gives $v_0\ne v_1$, so the all-zero input has
sensitivity $n$; hence $\bs(f;w)\le n=s(f)$.  The case $b=n-1$ is symmetric.

If $L=1$, the interval is the singleton $[w,w]$.  Maximality implies that
every one-bit change from an input of weight $w$ changes the function value
(the one existing side is sufficient at $w=0$ or $w=n$).  Thus $s(f)=n$,
and again $\bs(f;w)\le s(f)$.

It remains to consider a maximal interval meeting an endpoint.  For
$[0,b]$,
\[
 \bs(f;w)=\left\lfloor\frac{n-w}{b-w+1}\right\rfloor\le n-b.
\]
Since $v_b\ne v_{b+1}$, an input of weight $b$ has at least $n-b$ sensitive
coordinates, and hence $\bs(f;w)\le s(f)$.  The interval $[a,n]$ is
symmetric.  We have therefore proved, for every $w$ in a maximal constant
interval of length $L$,
\[
 \bs(f;w)\le\max\{s(f),n-L\}.                              \tag{**}
\]

Let $L_{\min}$ be the minimum length of a maximal constant interval of the
value vector.  By Lemma~\ref{lem:symmetric-certificate-runs},
\[
 C(f)=n-L_{\min}+1.
\]
Since $L\ge L_{\min}$, we have $n-L\le C(f)-1$.  Taking the maximum in
$(**)$ over all Hamming weights yields
\[
 \bs(f)\le\max\{s(f),C(f)-1\}.
\]
Finally, if $\bs(f)=C(f)$, then
\[
 C(f)\le\max\{s(f),C(f)-1\},
\]
so $s(f)\ge C(f)$.  Together with the standard inequalities
$s(f)\le\bs(f)\le C(f)$, this gives $s(f)=\bs(f)=C(f)$.
\end{proof}

\begin{corollary}[Only three patterns are possible]
\label{cor:three-patterns-symmetric}
For every nonconstant symmetric Boolean function, exactly one of the
following three mutually exclusive patterns occurs:
\[
 s=\bs=C,\qquad s=\bs<C,\qquad s<\bs<C.
\]
The fourth formally possible equality/strictness pattern $s<\bs=C$ is impossible.
\end{corollary}

\subsection{Application to elementary symmetric Boolean functions}

The preceding theorem makes the comparison between block sensitivity and
certificate complexity immediate once the equality case $s=C$ is known.

\begin{corollary}[Block sensitivity versus certificate complexity]
\label{cor:bs-certificate}
Let $1\le d\le n$.  If $d$ is odd, then
\[
 s(\sigma_{n,d})=\bs(\sigma_{n,d})=C(\sigma_{n,d})=n.
\]
Suppose that $d$ is even, and write
\[
 d=hD,\qquad h=2^{\nu_2(d)},\qquad D\text{ odd},
 \qquad n=hq+u,\quad 0\le u<h.
\]
Then
\[
 \bs(\sigma_{n,d})=C(\sigma_{n,d})
\]
if and only if
\[
 D\preceqtwo q\qquad\text{or}\qquad D\preceqtwo(q-1).
\]
In this case all three measures are equal.  If neither binary containment
holds, then
\[
 \bs(\sigma_{n,d})<C(\sigma_{n,d})=n-h+1.
\]
Thus, for even degree, equality at the upper end occurs exactly when the
boundary at Hamming weight $hq$ is a transition of the value sequence.
\end{corollary}

\begin{proof}
For odd $d$, Theorem~\ref{thm:odd} gives $s(\sigma_{n,d})=n$, while
\cite[Corollary~3.2]{ZhangLi2026} gives $C(\sigma_{n,d})=n$.  The standard
inequalities $s\le\bs\le C$ therefore give $s=\bs=C=n$.

Now let $d$ be even.  By \cite[Theorem~3.1]{ZhangLi2026},
\[
 C(\sigma_{n,d})=
 \begin{cases}
 n-u=hq,&D\preceqtwo q\text{ or }D\preceqtwo(q-1),\\
 n-h+1,&\text{otherwise}.
 \end{cases}
\]
If $D\preceqtwo q$ or $D\preceqtwo(q-1)$, then the boundary at $hq$ is a
transition.  Corollary~\ref{cor:end-transition} gives
\[
 s(\sigma_{n,d})=\max\{n-d+1,hq\}=hq=C(\sigma_{n,d}),
\]
because $hq-(n-d+1)=d-u-1\ge0$.  Hence $s\le\bs\le C$ yields
$s=\bs=C$.

Conversely, suppose that neither containment holds.  Then the same
certificate-complexity formula gives $C(\sigma_{n,d})=n-h+1$.  Moreover, failure of both containments implies $D>1$ and
$R_D(q)\le q-2$.  Hence
\[
 h\min\{q,R_D(q)+1\}\le h(q-1)<n-h+1=C(\sigma_{n,d}),
\]
while $d=hD>h$ gives
\[
 n-d+1<n-h+1=C(\sigma_{n,d}).
\]
The sensitivity formula therefore yields
$s(\sigma_{n,d})<C(\sigma_{n,d})$.  If one had
$\bs(\sigma_{n,d})=C(\sigma_{n,d})$, Theorem~\ref{thm:symmetric-bs-c} would
force $s(\sigma_{n,d})=C(\sigma_{n,d})$, a contradiction.  Therefore
$\bs(\sigma_{n,d})<C(\sigma_{n,d})$.
\end{proof}

\begin{remark}
In particular, every power-of-two degree satisfies
$\bs(\sigma_{n,2^k})=C(\sigma_{n,2^k})$, since $D=1$ and exactly one of
$q$ and $q-1$ is odd.  For even degrees with $D>1$, strict inequality occurs
precisely when the last two block indices $q-1$ and $q$ are both
non-supermasks of $D$.
\end{remark}

\begin{theorem}[Complete classification of the relations among $s$, $\bs$, and $C$]
\label{thm:three-pattern-classification}
Let $1\le d\le n$.  Exactly three equality/strictness patterns occur among
$s(\sigma_{n,d})$, $\bs(\sigma_{n,d})$, and $C(\sigma_{n,d})$, and they are classified as follows.

If $d$ is odd, then
\[
 s(\sigma_{n,d})=\bs(\sigma_{n,d})=C(\sigma_{n,d})=n.
\]
Now suppose that $d$ is even and write
\[
 d=hD,\qquad h=2^{\nu_2(d)},\qquad D\text{ odd},
 \qquad n=hq+u,\quad 0\le u<h.
\]
Put
\[
 \tau=\min\{q,R_D(q)+1\},\qquad
 S=\max\{n-d+1,h\tau\}=s(\sigma_{n,d}),
\]
and, when $q>D$, put
\[
 \rho=R_D(q-1),\qquad
 P=n-h+1-(h-1)D,\qquad Q=q+(h-1)\rho.
\]
Then:
\begin{enumerate}
\item[(I)] $s=\bs=C$ if and only if
\[
 D\preceqtwo q\quad\text{or}\quad D\preceqtwo(q-1).
\]
\item[(II)] $s=\bs<C$ if and only if
\[
 D\npreceqtwo q,\qquad D\npreceqtwo(q-1),
\]
and every applicable candidate satisfies
\[
 E_{n,d}\le S,\qquad P\le S,\qquad Q\le S,
\]
where $P,Q$ are present only when $q>D$.
\item[(III)] $s<\bs<C$ if and only if
\[
 D\npreceqtwo q,\qquad D\npreceqtwo(q-1),
\]
and at least one applicable candidate exceeds $S$; that is,
\[
 E_{n,d}>S\quad\text{or}\quad P>S\quad\text{or}\quad Q>S,
\]
where $P,Q$ are included only when $q>D$.
Equivalently, the more explicit arithmetic criterion of
Theorem~\ref{thm:strict-hierarchy} applies.
\end{enumerate}
The fourth formally possible equality/strictness pattern, $s<\bs=C$, does not occur.
Thus the three cases above are mutually exclusive and exhaustive.
\end{theorem}

\begin{proof}
For odd $d$, the first assertion follows from
Corollary~\ref{cor:bs-certificate}.  Suppose that $d$ is even.

If $D\preceqtwo q$ or $D\preceqtwo(q-1)$, the proof of
Corollary~\ref{cor:bs-certificate} gives $s=C$; hence the standard chain
$s\le\bs\le C$ gives $s=\bs=C$.  Conversely, if $s=\bs=C$, then in
particular $\bs=C$, so Corollary~\ref{cor:bs-certificate} gives one of the
two binary containments.  This proves (I).

Now assume that neither containment holds.  Corollary~\ref{cor:bs-certificate}
gives $\bs<C$.  Thus only the two patterns $s=\bs<C$ and $s<\bs<C$ remain.
By Corollary~\ref{cor:equality-separation}, the first occurs exactly when
every applicable candidate $E_{n,d},P,Q$ is at most $S$, and the second
occurs exactly when at least one applicable candidate exceeds $S$.  This
proves (II) and (III).  Finally, Theorem~\ref{thm:symmetric-bs-c} excludes
the remaining formal pattern $s<\bs=C$.  Hence (I)--(III) are mutually
exclusive and exhaustive.
\end{proof}

\begin{corollary}[An infinite family with strict separation]
\label{cor:bs-c-strict-family}
For every integer $m\ge 1$,
\[
 \bs(\sigma_{8m+2,6})=8m<8m+1=C(\sigma_{8m+2,6}).
\]
Consequently, there are infinitely many elementary symmetric Boolean
functions for which block sensitivity is strictly smaller than certificate
complexity.
\end{corollary}

\begin{proof}
Take $d=6=2\cdot 3$, so $h=2$ and $D=3=(11)_2$, and let
$n=8m+2$.  Then
\[
 n=2(4m+1),
 \qquad q=4m+1,
 \qquad u=0.
\]
Since $q\equiv1\pmod 4$ and $q-1\equiv0\pmod4$, neither $q$ nor
$q-1$ has its two least significant binary digits both equal to $1$.
Therefore
\[
 3\npreceqtwo q,
 \qquad
 3\npreceqtwo(q-1).
\]
Corollary~\ref{cor:bs-certificate} already implies strict inequality and gives
\[
 C(\sigma_{8m+2,6})=n-h+1=8m+1.
\]
To obtain the exact block sensitivity, note that the largest supermask of
$3$ not exceeding $q$ (and also not exceeding $q-1$) is
\[
 R_3(q)=R_3(q-1)=4m-1.
\]
The four candidates in Theorem~\ref{thm:bs-explicit} are therefore
\[
 I=8m-3,
 \qquad
 P=8m-2,
 \qquad
 Q=(4m+1)+(4m-1)=8m,
\]
and, for the final zero-run,
\[
 A=2\bigl(R_3(q)+1\bigr)=8m,
 \qquad
 E_{n,6}=\max\left\{8m,
 \left\lfloor\frac{8m+2}{3}\right\rfloor\right\}=8m.
\]
Hence
\[
 \bs(\sigma_{8m+2,6})=8m<8m+1=C(\sigma_{8m+2,6}),
\]
as claimed.
\end{proof}

\begin{theorem}[Complete strict hierarchy]\label{thm:strict-hierarchy}
Let $1\le d\le n$.  Then
\[
 s(\sigma_{n,d})<\bs(\sigma_{n,d})<C(\sigma_{n,d})
\]
holds if and only if $d$ is even and, on writing
\[
 d=hD,\qquad h=2^{\nu_2(d)},\qquad D\text{ odd},
 \qquad n=hq+u,\quad 0\le u<h,
\]
the following conditions hold.  First,
\begin{equation}\label{eq:strict-hierarchy-terminal}
 D\npreceqtwo q\qquad\text{and}\qquad D\npreceqtwo(q-1).
\end{equation}
Put
\[
 \rho=R_D(q-1)=R_D(q),\qquad g=q-\rho.
\]
Then at least one of the following two pairs of inequalities holds:
\begin{align}
 D&>h,
 & hg+u-2h+1&>(h-1)D, \label{eq:strict-hierarchy-P}\\
 g&>h,
 & hD-u-1&>(h-1)g. \label{eq:strict-hierarchy-Q}
\end{align}
Thus these binary and arithmetic conditions characterize all elementary
symmetric Boolean functions for which the three standard measures are
simultaneously strictly separated.
\end{theorem}

\begin{proof}
If $d$ is odd, then $s=\bs=C=n$, so strict separation is impossible.
Let $d$ be even.  By Corollary~\ref{cor:bs-certificate}, the inequality
$\bs<C$ is equivalent to \eqref{eq:strict-hierarchy-terminal}.  Under this
condition neither $q$ nor $q-1$ is a supermask of $D$, and hence
\[
 \rho=R_D(q-1)=R_D(q),\qquad g=q-\rho\ge2.
\]
Theorem~\ref{thm:main-s} gives
\[
 S:=s(\sigma_{n,d})=\max\{I,H\},
 \qquad I=n-d+1,\qquad H=h(\rho+1).
\]
The final zero-run begins at $A=h(\rho+1)=H$ and has length
$L=n-A+1$.  Its second endpoint candidate satisfies
\[
 \left\lfloor\frac{n}{L}\right\rfloor
 =1+\left\lfloor\frac{A-1}{L}\right\rfloor
 \le A=H.
\]
Therefore $E_{n,d}=H\le S$.  Thus, by
Theorem~\ref{thm:bs-explicit}, $\bs>S$ can occur only through the
first or last complete $1$-run candidates
\[
 P=n-h+1-(h-1)D,\qquad Q=q+(h-1)\rho.
\]

Now $P>S$ is equivalent to the simultaneous inequalities $P>I$ and $P>H$.
Straight calculation gives
\[
 P-I=D-h
\]
and
\[
 P-H=hg+u-2h+1-(h-1)D.
\]
Hence $P>S$ is exactly \eqref{eq:strict-hierarchy-P}.
Similarly, $Q>S$ is equivalent to $Q>I$ and $Q>H$, while
\[
 Q-H=g-h
\]
and
\[
 Q-I=hD-u-1-(h-1)g.
\]
Hence $Q>S$ is exactly \eqref{eq:strict-hierarchy-Q}.  Combining these
observations with the characterization of $\bs<C$ proves the theorem.
\end{proof}

\begin{corollary}[An infinite full-separation family]\label{cor:full-separation-family}
For every integer $m\ge1$,
\[
 \boxed{
 s(\sigma_{8m+4,6})=8m
 <8m+1=\bs(\sigma_{8m+4,6})
 <8m+3=C(\sigma_{8m+4,6}).}
\]
Consequently the strict hierarchy $s<\bs<C$ occurs for infinitely many
elementary symmetric Boolean functions.
\end{corollary}

\begin{proof}
Let $d=6=2\cdot3$, so $h=2$ and $D=3$, and put $n=8m+4$.
Then
\[
 q=4m+2,\qquad u=0,\qquad
 R_3(q)=R_3(q-1)=4m-1,
\]
since the supermasks of $3=(11)_2$ are precisely the integers congruent to
$3$ modulo $4$.  Hence $\tau=4m$ and Theorem~\ref{thm:main-s} gives
\[
 s(\sigma_{8m+4,6})
 =\max\{8m-1,8m\}=8m.
\]
The block-sensitivity candidates are
\[
 I=8m-1,\qquad P=8m,\qquad Q=8m+1.
\]
The final zero-run starts at $A=8m$ and has length $5$, so
\[
 E_{8m+4,6}
 =\max\left\{8m,\left\lfloor\frac{8m+4}{5}\right\rfloor\right\}=8m.
\]
Therefore Theorem~\ref{thm:bs-explicit} gives
$\bs(\sigma_{8m+4,6})=8m+1$.  Finally, neither $q$ nor $q-1$ is a
supermask of $3$, so \cite[Theorem~3.1]{ZhangLi2026} gives
\[
 C(\sigma_{8m+4,6})=n-h+1=8m+3.
\]
This proves the asserted strict hierarchy.
\end{proof}

\section{Relation with earlier formulas}

The binary-block viewpoint contains several earlier results as immediate
special cases.

For $d=2^k$, we have $D=1$ and $h=2^k$.  The block values alternate
between a zero-block and a one-block, so
Corollary~\ref{cor:power2} gives the previously known sensitivity formula.

For $d=6$, the example above gives an explicit
residue-class calculation showing that the old formula is recovered exactly.
For $d=10$ and $d=12$, the earlier calculations were organized by residue
classes modulo $16$.  In the present framework these residue classes arise
automatically from the binary containment condition on $D$ together with
the block length $h$, so the earlier formulas follow by the same calculation
without constructing each full periodic value string separately.

More generally, Theorem~\ref{thm:main-s} applies to every even degree,
including degrees having arbitrarily many nonzero binary digits.  Thus the
binary decomposition
\[
 d=2^{\nu_2(d)}D
\]
separates the two roles cleanly: the power of two determines the Hamming
weight block length, while the odd part determines the positions of the
$1$-blocks.

\section{Conclusion}

We have completely determined the sensitivity, average sensitivity, and
block sensitivity of the elementary symmetric Boolean functions
$\sigma_{n,d}$ for every $1\le d\le n$.  The three calculations are governed
by the same binary structure.  Writing
\[
 d=2^{\nu_2(d)}D,\qquad D\ {\rm odd},
\]
Lucas' theorem shows that the Hamming-weight value sequence is constant on
blocks whose length is determined by $2^{\nu_2(d)}$, while the odd part $D$
determines the locations of the $1$-blocks through binary containment.

For sensitivity, this gives a uniform exact formula for every even degree and,
together with the known odd-degree case, a complete determination for
arbitrary degree.  The same transition structure yields an exact
average-sensitivity formula through sensitive-edge counting.  For block
sensitivity, the previously known run-based theory for symmetric Boolean
functions combines with the binary structure to reduce the global maximum to
at most four explicit candidates: the initial run, the final run, the first
complete $1$-run, and the last complete $1$-run.

A further structural consequence concerns the relation among sensitivity,
block sensitivity, and certificate complexity for symmetric Boolean functions
in general.  We prove that every nonconstant symmetric Boolean function $f$
satisfies
\[
 \bs(f)\le \max\{s(f),C(f)-1\}.
\]
Thus $\bs(f)=C(f)$ forces $s(f)=\bs(f)=C(f)$, and the equality pattern
$s(f)<\bs(f)=C(f)$ is impossible in the symmetric class.

Combining this general theorem with the arbitrary-degree
certificate-complexity formula in \cite{ZhangLi2026} yields a complete
classification for elementary symmetric Boolean functions.  Exactly three
patterns are realizable:
\[
 s(\sigma_{n,d})=\bs(\sigma_{n,d})=C(\sigma_{n,d}),
\]
\[
 s(\sigma_{n,d})=\bs(\sigma_{n,d})<C(\sigma_{n,d}),
\]
and
\[
 s(\sigma_{n,d})<\bs(\sigma_{n,d})<C(\sigma_{n,d}).
\]
Necessary and sufficient conditions are obtained for each pattern.  In
particular, for even $d=hD$ and $n=hq+u$, equality
$s=\bs=C$ holds exactly when $D\preceqtwo q$ or
$D\preceqtwo(q-1)$; otherwise $\bs<C$, and the explicit block-sensitivity
candidates decide whether $s=\bs$ or $s<\bs$.  The strict hierarchy is
realized infinitely often; for every $m\ge1$,
\[
 s(\sigma_{8m+4,6})=8m<8m+1=\bs(\sigma_{8m+4,6})
 <8m+3=C(\sigma_{8m+4,6}).
\]

Thus the binary containment structure supplied by Lucas' theorem provides a
single combinatorial framework for the full elementary-symmetric family.  It
replaces separate periodic-string and residue-class calculations for
individual degrees, yields exact formulas for the principal sensitivity
parameters, and determines all equality and strict-separation patterns among
$s$, $\bs$, and $C$.

\section*{Declaration of Generative AI and AI-assisted technologies in the manuscript preparation process}

During the preparation of this work, the authors used ChatGPT (OpenAI) to
assist with the organization and presentation of the manuscript and with the
examination of mathematical arguments.  The authors independently verified
all mathematical results and proofs, reviewed and edited the AI-assisted
content, and take full responsibility for the content of the article.

\end{document}